\documentclass[10pt]{amsart}
\usepackage{amssymb}
\usepackage{epsfig}
\usepackage{url}
\usepackage{setspace}
\usepackage{pdflscape}
\usepackage{color}
\theoremstyle{plain}

\newtheorem{thm}{Theorem}[section]
\newtheorem{cor}[thm]{Corollary}
\newtheorem{lem}[thm]{Lemma}

\newtheorem{rem}[thm]{Remark}
\newtheorem{ques}[thm]{Question}
\newtheorem{conj}[thm]{Conjecture}

\newtheorem{prob}[thm]{Problem}
\theoremstyle{definition}
\newtheorem*{acknowledgement}{Acknowledgement}
\def\cal{\mathcal}
\def\bbb{\mathbb}
\def\op{\operatorname}
\renewcommand{\phi}{\varphi}

\newcommand{\N}{\bbb{N}}
\newcommand{\Z}{\bbb{Z}}
\newcommand{\Q}{\bbb{Q}}

\makeatletter
\let\@@pmod\pmod
\DeclareRobustCommand{\pmod}{\@ifstar\@pmods\@@pmod}
\def\@pmods#1{\mkern4mu({\operator@font mod}\mkern 6mu#1)}
\makeatother

\begin{document}

\title[Equal values of partition functions via Diophantine equations]{Equal values of certain partition functions via Diophantine equations}
\author{Szabolcs Tengely}
\address{Szabolcs Tengely\newline
\indent Mathematical Institute\newline
 \indent University of Debrecen\newline
 \indent P.O.Box 12\newline
 \indent 4010 Debrecen\newline
 \indent Hungary}
\email{tengely@science.unideb.hu}

\author{Maciej Ulas}
\address{
	Maciej Ulas\newline
	\indent Jagiellonian University\newline
	\indent Faculty of Mathematics and Computer Science \newline
	\indent Institute of Mathematics\newline
	\indent \L{}ojasiewicza 6\newline
	\indent 30-348 Krak\'ow\newline
	\indent Poland}
\email{Maciej.Ulas@im.uj.edu.pl}

\keywords{partitions, Diophantine equation, polynomial, identities}
\subjclass[2000]{11D41, 11P83}
\thanks{
The research of the first author was supported in part by the NKFIH grants 115479, 128088 and130909 and by the project EFOP-3.6.1-16-2016-00022, co-financed by the European Union and the European Social Fund.  Research of the second author was supported by a grant of the National Science Centre (NCN), Poland, no. UMO-2019/34/E/ST1/00094}

\begin{abstract}
Let $A\subset \N_{+}$ and by $P_{A}(n)$ denotes the number of partitions of an integer $n$ into parts from the set $A$. The aim of this paper is to prove several result concerning the existence of integer solutions of Diophantine equations of the form $P_{A}(x)=P_{B}(y)$, where $A, B$ are certain finite sets.
\end{abstract}

\maketitle

\section{Introduction}\label{sec1}

Let $\N$ be the set of non-negative integers, $\N_{+}$ the set of positive integers and for $k\in\N_{+}$ we write $\N_{\geq k}$ for the set of integers $\geq k$.

Let $A\subset\N_{+}$ be given and take $n\in\N$. By an $A$-partition $\lambda=(\lambda_{1},\ldots, \lambda_{k})$, of a non-negative integer $n$ with parts in $A$, we mean representation of $n$ in the form
$$
n=\lambda_{1}+\ldots+\lambda_{k},
$$
where $\lambda_{i}\in A$. The representations of $n$ differing only in order of the terms are counted as a one. We also put
$$
\op{Part}_{A}(n)=\{\lambda:\;\lambda\;\mbox{is}\;A\mbox{-partition of}\;n\},
$$
and consider the corresponding partition function
$$
P_{A}(n):=\#\op{Part}_{A}(n).
$$

It is well know that the ordinary generating function of the sequence $(P_{A}(n))_{n\in\N}$ takes the form
$$
\prod_{a\in A}\frac{1}{1-x^{a}}=\sum_{n=0}^{\infty}P_{A}(n)x^{n}.
$$
In particular, if $A=\N_{+}$, then $P_{A}(n)$, simply denoted as $p(n)$, is the famous partition function introduced by L. Euler and extensively studied by S. Ramanujan. The function $p(n)$ counts the number of partitions with parts in $\N_{+}$, i.e., unrestricted partitions of $n$. Many questions concerning arithmetic properties of $p(n)$ is still unsolved. Probably, the most famous one, is the question whether the sequence $\{p(n)\mod{m}\}_{n\in\N}$ is equidistributed modulo $m$, i.e., whether, for any given $m\in\N_{\geq 2}$ and each $r\in\{0,\ldots, m-1\}$, there is an equality
$$
\limsup_{N\rightarrow  +\infty}\frac{\#\{n\leq N:\;p(n)\equiv r\pmod*{m}\}}{N}=\frac{1}{m}.
$$
For discussion on this topic see \cite{AO}.

The literature concerning arithmetic properties of functions counting various partitions is enormous. One of the central problems in partition theory is to obtain partition identities \cite{And1}. More precisely, if $\cal{W}$ is a certain property which can be applied to the parts of a given $A$-partition $\lambda$ of a positive integer $n$, then by $P_{A}(\cal{W},n)$ we denote the number of $A$-partitions of $n$ which have the property $\cal{W}$. Thus, by a partition identity we mean an identity of the form
$$
P_{A_{1}}(\cal{W}_{1},n)=P_{A_{2}}(\cal{W}_{2},n),
$$
where $A_{1}, A_{2}\subset \N_{+}$ and $\cal{W}_{1}, \cal{W}_{2}$ are given properties. The basic identity of the kind is Euler's famous identity 
$$
\prod_{n=1}^{\infty}(1+x^{n})=\prod_{n=1}^{\infty}\frac{1}{1-x^{2n-1}},
$$
which gives the equality between the number of partitions into distinct parts and the number of partitions into odd parts. Equivalently
$$
P_{\N_{+}}(\mbox{distinct parts},n)=P_{\N_{+}}(\mbox{odd parts},n).
$$
Euler's identity can be proved by a simple manipulation of infinite products. Much deeper partition identities are two (of many known) Rogers-Ramanujan type identities. These identities can be deduced from the equalities of the type: infinite product $=$ infinite series, i.e.,
\begin{align*}
\prod_{n=1}^{\infty}\frac{1}{(1-x^{5n-1})(1-x^{5n-4})}&=\sum_{n=1}^{\infty}\frac{x^{n^{2}}}{(1-x)(1-x^{2})\cdot\ldots\cdot (1-x^{n})},\\
\prod_{n=1}^{\infty}\frac{1}{(1-x^{5n-2})(1-x^{5n-3})}&=\sum_{n=1}^{\infty}\frac{x^{n^{2}+n}}{(1-x)(1-x^{2})\cdot\ldots\cdot (1-x^{n})}.
\end{align*}
The first identity implies that the number of partitions of $n$ such that the adjacent parts differ by at least 2 is the same as the number of partitions of $n$ such that each part is congruent to either 1 or 4 modulo 5, i.e.,
$$
P_{\N_{+}}(\mbox{adjacent parts differ by at least 2},n)=P_{\N_{+}}(\mbox{parts}\equiv 1, 4\pmod*{5},n).
$$ 
The second one implies that the number of partitions of $n$ such that the adjacent parts differ by at least 2 and such that the smallest part is at least 2 is the same as the number of partitions of $n$ such that each part is congruent to either 2 or 3 modulo 5. Mentioned identities are only special cases of more general one which can be found in the literature. However, according to our best knowledge nothing is known about identities of the form
\begin{equation}\label{eq1}
P_{A_{1}}(\cal{W}_{1},n)=P_{A_{2}}(\cal{W}_{2},m)\neq 0,
\end{equation}
where $n, m$ are {\it different} positive integers.

Of course, in the above, we are interested in non-trivial identities. More precisely, if the set of values of $P_{A_{i}}(\cal{W}_{i},n)$ contains all positive integers for some $i\in\{1,2\}$ then the problem is trivial. For example, if $A_{1}=\{2^{j}:\;j\in\N\}$ and the property $\cal{W}_{1}$ says that at least two parts in $A_{1}$-partition of $n$ are equal, then, for $s(n):=P_{A_{1}}(\cal{W}_{1},n)$ we clearly have
$$
\sum_{n=0}^{\infty}s(n)x^{n}=\prod_{n=0}^{\infty}(1+x^{2^{n}}+x^{2^{n+1}}),
$$
and the sequence $\{s(n)\}_{n\in\N}$ is a famous Stern sequence satisfying the recurrence relations
$$
s(0)=1,\quad s(2n)=s(n)+s(n-1),\quad s(2n+1)=s(n).
$$
One can easily prove that $s(2^{n})=n+1$ and thus the sequence $(s(n))_{n\in\N}$ contains all positive integers. Thus the question concerning the existence of positive integer solutions of related equation is trivial.

Of great interests would be the proof of non-existence of $A_{i}, \cal{W}_{i}, i=1, 2$ such that the corresponding partition functions $P_{A_{i}}(\cal{W}_{i},n), i=1, 2,$ have exponential growth and  equation (\ref{eq1}) has infinitely many solutions in positive integers. Note that if $P_{A_{i}}(\cal{W}_{i},n)$ has an exponential growth then the set $A_{i}$ is necessarily infinite.

A related question is to whether given partition function takes values in a given infinite set. Especially, in the set of values of a given polynomial. In other words we are interested in the solvability in non-negative integers of the equation
\begin{equation}\label{eq2}
f(m)=P_{A}(\cal{W},n)\neq 0,
\end{equation}
where $f\in\Q[x]$ is of positive degree and positive leading coefficient. Again we are interested in non-trivial situations only. Here if $f$ is linear, then we enter in the realm of partition congruences. The Ramanujan congruences
$$
p(5n+4)\equiv 0\pmod*{5},\quad p(7n+5)\equiv 0\pmod*{7},\quad p(11n+7)\equiv 0\pmod*{11}
$$
and its various generalizations give some non-trivial examples when equation (\ref{eq2}) has infinitely many solutions in positive integers. Indeed, each equation
$$
p(n)=5m, \quad p(n)=7m,\quad p(n)=11m
$$
has infinitely many solutions in positive integers. In fact, one can prove that for each prime $q\neq 3$ the Diophantine equation $p(n)=qm$ has infinitely many solutions in positive integers. Indeed, Nicolas, Ruzsa and S\'{a}rk\"{o}zy \cite{NRS} proved this for $q=2$, and Ono proved that such a statement is true for each $q\geq 5$ \cite{Ono}.

In case of equation (\ref{eq2}) more can be said provided we know some arithmetic properties of the partition function $P_{A}(\cal{W},n)$. For example, if $A=\{2^{i}:\;i\in\N\}$ then so called binary partition function $b(n):=P_{A}(n)$ counting the partitions with all parts being powers of $2$, satisfies the recurrence relation
$$
b(0)=1,\quad b(2n)=b(2n-1)+b(n),\quad b(2n+1)=b(2n).
$$
A classical result of Churchhouse states that $\nu_{2}(b(n))\in\{1,2\}$ for $n\in\N_{\geq 2}$ \cite{Chu}. Thus, if
$$
\{1,2\}\cap\{\nu_{2}(f(m)):\;m\in\N_{+}\}=\emptyset
$$
then equation (\ref{eq2}) has at least $\op{deg} f$ solutions possibly coming from the integer solutions of the equation $f(m)=1$.

Our discussion above shows that the problem of solvability of equation (\ref{eq1}) or equation (\ref{eq2}) is interesting, difficult and worth of further investigations. It is clear that our questions are on the intersection of discrete mathematics, combinatorics and Diophantine equations. This suggest to start investigations with the case of $A$ finite. This is reasonable due to the fact that $p_{A}(n)$ is a quasi polynomial (see for example \cite{RS}). More precisely, if $A=(a_{1},\ldots, a_{k})$ and $L_{A}:=\op{lcm}(a_{1}, \ldots, a_{k})$, then
$$
p_{A}(L_{A}n+i)\in\Q[n]\quad\mbox{for}\quad i=0, 1, \ldots, L_{A}-1.
$$
Thus, if $A_{1}, A_{2}$ are finite, then solvability in positive integers of equation (\ref{eq1}) or a more general equation (\ref{eq2}), is equivalent with the solvability in non-negative integers, of at least one equation of the type
$$
p_{A_{1}}(L_{A_{1}}m+i)=p_{A_{2}}(L_{A_{2}}n+j),
$$
where $i\in\{0, 1, \ldots, L_{A_{1}}-1\}, j\in\{0, 1, \ldots, L_{A_{2}}-1\}$, or
$$
f(m)=p_{A_{1}}(L_{A_{1}}n+i),\quad i\in\{0, 1, \ldots, L_{A_{1}}-1\},
$$
respectively. Thus, we enter into realm of polynomial Diophantine equations with separable variables where great deal of methods are for our disposal.

In the rest of the paper, in case of the set $A=\{1,\ldots, m\}$ instead of writing $P_{A}(x)$ and $\op{Part}_{A}(n)$, we will simply write $P_{m}(x)$ and $\op{Part}_{m}(n)$, respectively. Let us describe the content of the paper in some details.

In Section \ref{sec2} we prove that for each $f\in\Z[x]$ with positive degree and positive leading coefficient and $A=\{a_{1}, a_{2}\}\subset \N_{+}$ with $\gcd(a_{1},a_{2})=1$, the Diophantine equation $P_{A}(x)=f(y)$ has infinitely many solutions in positive integers (Theorem \ref{a1a2}).

Section \ref{sec3} is mainly devoted to the study of the equation $P_{3}(x)=P_{n}(y)$ for $n=4, 5$. In particular, we describe all positive integer solutions in both cases and present some related results (Theorem \ref{P3P4} and Theorem \ref{P3P5}).

In Section \ref{sec4} we study some equation involving $P_{A}(x)$, where $A=\{1,2,a\}, a\in\N_{\geq 3}$. In particular, we obtain a general result concerning the existence of infinitely many positive integer solutions of the equation $P_{A}(x)=P_{4}(y)$ (Theorem \ref{PA4}). We also obtain, under weak assumptions on $a, b\in\N_{\geq 3}, a\neq b$, that for $A=\{1, 2, a\}, B=\{1, 2, b\}$, the Diophantine equation $P_{A}(x)=P_{B}(y)$ has infinitely many solutions in positive integers (Theorem \ref{12a12b}).

In Section \ref{sec5} we obtain several results concerning the square values of $P_{A}(x)$. In particular we describe the set of positive integer solutions of the equation $y^2=P_{n}(x)$ for $n=3, 4, 5$. We also discuss results of some computations. Finally, in the last section we collect some general questions and conjectures concerning various aspects of the Diophantine equations of the form $P_{A}(x)=P_{B}(y)$ and report results of various computations.

\section{The case of $A=\{a_{1},a_{2}\}$}\label{sec2}

In this short section we prove a general result concerning the existence of positive integer solutions of the Diophantine equation
$$
P_{A}(x)=f(y),
$$
where $A=\{a_{1}, a_{2}\}, a_{1}<a_{2}$ and $f\in\Z[y]$ is a non-constant polynomial with positive leading coefficient. More precisely, based on the formula obtained by Sert\"{o}z in \cite{Se} we easily prove the following.

\begin{thm}\label{a1a2}
Let $A=\{a_{1}, a_{2}\}\subset\N_{+}$. Then, for each $f\in\Z[x]$ with positive leading coefficient, the Diophantine equation $P_{A}(x)=f(y)$ has infinitely many solutions in positive integers.
\end{thm}
\begin{proof}
Without loss of generality we can assume that $\gcd(a_{1},a_{2})=1$. Then, by Sert\"{o}z result \cite{Se}, we know that the following formula holds
$$
P_{A}(n)=\frac{n+a_{1}\cdot a_{1}'+a_{2}\cdot a_{2}'}{a_{1}a_{2}}-1,
$$
where $a_{i}\cdot a_{i}'\equiv-n\pmod*{a_{i+1}}$ and $1\leq a_{i}'(n)\leq a_{i+1}$ for $i=1, 2$ and if $i=2$, then $i+1$ is taken modulo 2. In particular, for each $n\in\N_{+}$
$$
P_{A}(n)=\left\lfloor\frac{n}{a_{1}a_{2}}\right\rfloor\quad\mbox{or}\quad P_{A}(n)=\left\lfloor\frac{n}{a_{1}a_{2}}\right\rfloor +1.
$$
Thus, let us take $n=a_{1}a_{2}(f(m)-1)$, with $m\in\N_{+}$ chosen in such a way that $f(m)>1$. We thus consider the congruence
$$
a_{i}\cdot a_{i}'\equiv-n\equiv -a_{1}a_{2}(f(m)-1)\equiv 0\pmod*{a_{i+1}}
$$
and due to co-primality condition and the bound for $a_{1}', a_{2}'$, we get that $a_{i}'=a_{i+1}$, i.e., we have that $a_{1}'=a_{2}, a_{2}'=a_{1}$. In consequence
$$
P_{A}(a_{1}a_{2}(f(m)-1))=\frac{a_{1}a_{2}(f(m)-1)+a_{1}a_{2}+a_{2}a_{1}}{a_{1}a_{2}}-1=f(m)
$$
and our theorem is proved.
\end{proof}

\section{The equation $P_{3}(x)=P_{n}(y)$ for $n=4, 5$}\label{sec3}

In this section we are interesting in the characterization of positive integer solutions of the Diophantine equation
$$
P_{3}(x)=P_{n}(y)
$$
for $n=4, 5$.

\begin{thm}\label{P3P4}
The Diophantine equation $P_{3}(x)=P_{4}(y)$ has infinitely many solutions in integers.
\end{thm}
\begin{proof}
First of all we recall that
\begin{align*}
P_{3}(n)&=\left\lfloor\frac{(n+3)^{2}}{12}\right\rceil,\\
P_{4}(n)&=\left\lfloor\frac{(n+1)(n^2+23n+85)}{144}-\frac{n+4}{8}\left\lfloor\frac{n+1}{2}\right\rfloor\right\rceil,
\end{align*}
where $\lfloor x\rceil$ denotes the nearest integer to $x$. For concise proofs of these two identities see \cite[p. 57-60]{AE}. An alternative recent proof of these and related equalities were obtained by Castillo et al. \cite{Cas}.

From the form of $P_{3}(n)$ we see that for $i\in\{0,\ldots, 5\}$ the expression $P_{3}(6n+i)$ is a polynomial in $n$. More precisely, we define $P_{i,6,3}(n)=P_{3}(6n+i)$ and observe that
\begin{equation*}
\begin{array}{ll}
  P_{0,6,3}(n)=3n^2+3n+1,      & P_{1,6,3}(n)=(n+1)(3n+1),  \\
  P_{2,6,3}(n)=(n+1)(3n+2),    & P_{3,6,3}(n)=3(n+1)^2,\\
  P_{4,6,3}(n)=(n+1)(3n+4),    & P_{5,6,3}(n)=(n+1)(3n+5).\\
\end{array}
\end{equation*}
We can similar treat the case of $P_{4}(n)$ and by defining $P_{2i+1,6,4}(n)=P_{4}(6n+2i+1)$ for $i=0, 1, 2$ and $P_{2i,12,4}(n)=P_{4}(12n+2i)$ for $i=0, 1, \ldots, 5$, we get
\begin{equation*}
\begin{array}{ll}
  P_{1,6,4}(n)=\frac{1}{2} (n+1) \left(3 n^2+6 n+2\right),     & P_{3,6,4}(n)=\frac{3}{2} (n+1)^2 (n+2), \\
  P_{5,6,4}(n)=\frac{3}{2} (n+1) (n+2)^2, &  \\
  P_{0,12,4}(n)=12n^3+15n^2+6n+1,       & P_{2,12,4}(n)=12n^3+21n^2+12n+2, \\
  P_{4,12,4}(n)=(n+1)(12n^2+15n+5),     & P_{6,12,4}(n)=3(n+1)^2(4n+3), \\
  P_{8,12,4}(n)=3(n+1)^2(4n+5),         & P_{10,12,4}(n)=(n+1)(12n^2+33n+23).
\end{array}
\end{equation*}
In order to characterize all integers solutions of the Diophantine equation $P_{3}(x)=P_{4}(y)$ we need to perform case by case analysis. More precisely, we consider all possible combinations of the equations
\begin{equation*}\label{type12}
{\rm (I)}\quad P_{i,6,3}(x)=P_{2j+1,6,4}(y)\quad\quad\mbox{or}\quad\quad {\rm(II)}\quad P_{i,6,3}(x)=P_{2j,12,4}(y),
\end{equation*}
i.e., we deal with 54 equations. In each case we deal in the same way. Because the degree of $P_{i,6,3}$ is 2, each equation of interests can be reduced to the equation of the type $Y^2=f(X)$ for some $f\in\Z[X]$ and the degree of $f$ is 3. If $f$ has no multiple roots, by classical Siegel result, we know that the curve defined by the equation $Y^2=f(X)$ has only finitely many integral points. On the other hand, if $f$ has multiple roots then there is a chance that our equation has infinitely many integral solutions which can be parameterized via polynomials.

In order to see what is going on, let us consider the equation \begin{equation}\label{063064}
P_{1,6,3}(x)=P_{3,6,4}(y),\;\mbox{i.e.},\; (x+1)(3x+1)=\frac{3}{2}(y+1)^2(y+2)
\end{equation}
or equivalently $Y^2=X^3-108X+1728$, where we put $X=6(3y+4), Y=36(3x+2)$. Our equation represents an elliptic curve, say $E$, in the plane $(X,Y)$ and standard methods allow to find that the curve $E$ has trivial torsion and the rank of $E$ is equal to 2, with the generators $(X,Y)=(6, -36), (-2, 44)$. Using internal {\sc Magma} \cite{Mag} procedures

\quad {\tt E:=EllipticCurve([-108,1728]);}

\quad {\tt IntegralPoints(E);}\\
(the background for this latter routine is found in \cite{GPZell} and \cite{StTz}; see also \cite{Tz1996})
we find that the point $(X,Y)\in E(\Q)$ has integer coordinates, if and only if
\begin{align*}
(X,Y)\in\{ &(-12, \pm 36), (-3, \pm 45), (-2, \pm 44), (6, \pm 36), (16, \pm 64),  \\
           &(22, \pm 100), (78, \pm 684), (96, \pm 936), (7926, \pm 705636)\}.
\end{align*}
Direct check shows that only the points $P_{1}=(96, -936)$ and $P_{2}=(7926, -705636)$ correspond to the solutions of our equation. We thus get that the only integer solutions of equation (\ref{063064}) are $(x,y)=(8, 4), (6533, 439)$. We thus get the equalities
\begin{align*}
P_{1,6,3}(8)&=P_{3,6,4}(4)=225,\\
P_{1,6,3}(6533)&=P_{3,6,4}(439)=128066400.
\end{align*}

Case by case analysis reveals that the solutions exist only in the following cases:
\begin{align*}
(i,j)&=(0, 0), (1, 0), (1, 1), (1, 2), (5, 1)\quad\mbox{in the case (I)}\quad \mbox{and}\\
(i,j)&=(0,0), (1,0), (2,1), (3,4), (5,2)\quad\mbox{in the case (II)}.
\end{align*}

In the table below we give all integral solutions in these cases.
\begin{equation*}
\begin{array}{|l|l|}
\hline
\hline
  (i,j) & \mbox{integral solutions}\;(x,y)\;\mbox{of}\; P_{i,6,3}(x)=P_{2j+1,6,4}(y)  \\
\hline
\hline
  (0, 0) &  (0, 0)\\
  (0, 1) &  (0, 0) \\
  (1, 1) &  (8, 4), (6533, 439)\\
  (1, 2) &  (293, 54) \\
  (3, 1) &  ((t-1)(2 t^2+2 t+1), 2(t-1)(t+1)), t\in\N_{+} \\
  (3, 2) &  (2 t^3+t-1, 2 t^2-1), t\in\N_{+} \\
  (5, 1) &  (5, 3)\\
\hline
\hline
  (i,j) & \mbox{integral solutions}\;(x,y)\;\mbox{of}\; P_{i,6,3}(x)=P_{2j,12,4}(y) \\
\hline
\hline
  (0, 0) & ((t-1)(2 t^2-t+1), 2 (t-1) t), t\in\N_{+}\\
  (1, 0) &  (0, 0) \\
  (2, 1) &  (0, 0) \\
  (3, 4) & (2 t^3+3 t^2+t-1, t^2+t-1), t\in\N_{+}\\
  (5, 2) & (0, 0) \\
\hline
\end{array}
\end{equation*}
\begin{center} Table. Values of $(i,j)$ such that the corresponding equations of types (I), (II) have solutions in non-negative integers. \end{center}

Having the form of solutions presented in the table above, we can easily back to our original Diophantine equation $P_{3}(X)=P_{4}(Y)$ and found that solutions take the form $(X,Y)=(6x+i, 6y+2j)$ in the first type equation, and $(X,Y)=(6x+i, 12y+2j+1)$ in the second type equation.
\end{proof}

In our next theorem we characterize the set of positive integer solutions of the Diophantine equation $P_{3}(x)=P_{5}(y)$.

\begin{thm}\label{P3P5}
The equation $P_{3}(x)=P_{5}(y)$ has only finitely many solutions in positive integers. More precisely, the pair $(x,y)$ is a solution if and only if $(x,y)\in \cal{A}$, where
\begin{align*}
\cal{A}=\{&(1, 1), (2, 2), (3, 3), (5, 4), (6, 5), (8, 6), (16, 10), (18, 11), (26, 14), \\
          &(45, 20), (174, 45), (217, 51), (457, 77), (468, 78), (701, 97), (10093, 388)\}.
\end{align*}
\end{thm}
\begin{proof}
Direct check reveals that $P_{5}(60n+i), i\in\{0,\ldots, 59\}$ is a polynomial in variable $n$. Thus, one can perform the same analysis as in the case of the equation $P_{3}(x)=P_{4}(y)$. However, here the situation is a bit more complicated because, after necessary simplifications, we need to work with the equations of the type $Y^2=f(X)$, where $f$ is a polynomial of degree 4. We need to consider $6\cdot 60=360$ equations in order to get the result. Here we may apply the {\sc Magma} procedure \texttt{IntegralQuarticPoints()} based on the paper \cite{Tz1996}.
It worked well in all except the 8 cases, where the {\sc Magma} function failed to determine the complete set of integral solutions. These 8 problematic equations are of the form
$$
P_{3}(6y+i)=P_{5}(60x+j)
$$
for
$$
(i,j)\in\cal{A}=\{(3,9), (3, 12), (3, 21), (3, 24), (3, 33), (3, 36), (3, 48), (3, 57)\}.
$$

The equations corresponding to $(i, j)=(3, 48), (3, 57)$ are of the following form
\begin{eqnarray*}
Y^2 &=& u(54000u^3 - 16200u^2 + 1410u - 18),\\
Y^2 &=& u(54000u^3 + 16200u^2 + 1410u + 18),
\end{eqnarray*}
respectively, where $u=x+1.$ In both cases we obtain that $u$ is a square multiplied by a divisor of 18. Therefore we need to handle the equations
$$
(2\delta^2 v)^2=(60\delta u)^3 - 18\delta(60\delta u)^2 + 94\delta^2(60\delta u) - 72\delta^3,
$$
where $\delta\in\{\pm 1, \pm 2, \pm 3, \pm 6, \pm 9, \pm 18\}.$ One more time we use the MAGMA procedure \texttt{IntegralPoints()} to determine the integral points on these elliptic curves. We only need to consider points having first coordinate divisible by $60\delta.$ It turns out that $u=0$ is the only solution, that is $(x,Y)=(-1,0).$

In the remaining 6 cases,  we observed that the discriminant of $P_{3}(6y+i)=P_{5}(60x+j)$ with respect to $y$ is equal to
$$
F(u)=432u^4 + 648u^3 + 282u^2 + 18u,
$$
for suitable substitution of the form $u=ax+b$ (depending on values of $i, j$). The expression for $u$ are given below
$$
\begin{array}{|l|l|l|l|l|l|l|}
\hline
  (i,j) & (3, 9) & (3, 12) & (3, 21) & (3, 24) & (3, 33) & (3,36) \\
\hline
      u & 5x+1 & -5x-2 & 5x+2 & -5x-3 & 5x+3 & -5x-4\\
\hline
\end{array}
$$

Therefore we only need to determine integral points on the curve
$$
72u^4 + 108u^3 + 47u^2 + 3u=30v^2.
$$
We obtain that 3 divides $u,$ so $u=3u_1$ for some integer $u_1.$ We have that
$$
u_1(648u_1^3+324u_1^2+47u_1+1)=30v_1^2,\mbox{ where }v=3v_1.
$$
The factorization yields the following elliptic curves
$$
X^3+324\delta X^2+30456\delta^2X+419904\delta^3=Y^2, \mbox{ where }\delta\in\{1,2,3,5,6,10,15,30\}.
$$
We determined the integral points on these curves and checked if $X$ is divisible by $648\delta,$ the only such solution corresponds to $X=0.$ Hence we do not obtain integral solution in case of these six curves.
\end{proof}

In the light of the result one can ask for which sequences $A$ of the form $A=\{1, 2, 3, a\}, a\geq 4$, the Diophantine equation $P_{3}(x)=P_{A}(y)$ has infinitely many solutions
in positive integers. It is not difficult to find many values of $a$ with this property. Indeed, for a fixed $a$ the equation has the form that a quadratic polynomial is equal to a cubic polynomial, hence we may expect a genus 1 curve. However, we for certain values of $a$ we may obtain infinitely many integral solutions. The strategy we follow is simple, we determine polynomials $P_A(6an+k)$ in $n$ that are not square-free and then deal with the equation $P_{3}(6m+3)=3(m+1)^2=P_A(6an+k)$. This works for $a\in\{4, 6, 12, 14, 20\}$. In these cases the equation $P_{3}(x)=P_{A}(y)$ has a polynomial solution and hence infinitely many solutions in positive integers.

We close this section with the following

\begin{conj}
There are infinitely many values of $a\in\N_{\geq 4}$ such that for $A=\{1, 2, 3, a\}$, the Diophantine equation $P_{3}(x)=P_{A}(y)$ has infinitely many solutions in positive integers.
\end{conj}

\section{Some properties of $P_{A}(x)$ for $A=\{1,2,a\}, a\geq 3$, and related equations}\label{sec4}

In this section we obtain explicit expression for $P_{A}(n)$ in case of $A=\{1, 2, a\}$. As an application we deduce several results concerning Diophantine properties of $P_{A}(n)$.

\begin{thm}\label{12aexpres}
Let $a\in\N_{\geq 3}$ and put $A=\{1,2,a\}$. If $a=2c$ for some $c\in\N_{\geq 2}$ then
\begin{equation*}
P_{A}(4cn+i)=2cn^2+\left(c+2\left\lfloor\frac{i}{2}\right\rfloor+2\right)n+\begin{cases}\begin{array}{ll}
                                                                                    \left\lfloor\frac{i+2}{2}\right\rfloor, & i\in\{0, \ldots, 2c-1\}  \\
                                                                                    2\left\lfloor\frac{i}{2}\right\rfloor +2-a,& i\in\{2c, \ldots, 4c-1\}
                                                                                  \end{array}
\end{cases}.
\end{equation*}

If $a=2c+1$ for some $c\in\N_{+}$ then
\begin{align*}
P_{A}(2(2c+1)n&+i)=(2c+1)n^2+(c+i+2)n\\
              &+\begin{cases}\begin{array}{ll}
                                                                                    \left\lfloor\frac{i+2}{2}\right\rfloor, & i\in\{0, \ldots, 2c\}  \\
                                                                                    i+1-a,& i\in\{2c+1, \ldots, 4c+1\}
                                                                                  \end{array}
\end{cases}.
\end{align*}
\end{thm}
\begin{proof}
Let $A=\{1,2,a\}$ and recall that
$$
\frac{1}{(1-x)(1-x^2)(1-x^a)}=\sum_{n=0}^{\infty}P_{A}(n)x^{n}, \quad \frac{1}{(1-x)(1-x^2)}=\sum_{n=0}^{\infty}\left(\left\lfloor\frac{n}{2}\right\rfloor+1\right)x^{n}.
$$
Thus, using the identity $(1-x^{a})F_{A}(x)=\frac{1}{(1-x)(1-x^2)}$ by comparison of coefficients of like powers on both sides of we get that $P_{A}(n)$ satisfies the following recurrence relation:
$$
P_{A}(n)=\left\lfloor\frac{n}{2}\right\rfloor+1, n\leq a-1\quad\mbox{and}\quad P_{A}(n)=P_{A}(n-a)+\left\lfloor\frac{n}{2}\right\rfloor+1\;\mbox{for} \;n\geq a.
$$
Knowing that $P_{A}(n)$ satisfies recurrence relation of the presented form and using the (conjectural) form of the solution it is easy to perform the rest of the proof by induction. We omit the tiresome details.
\end{proof}

There are many papers devoted to the explicit computation of the function $P_{A}(n)$ for given $A=\{a_{1},\ldots, a_{k}\}$ under various conditions on $a_{1}, \ldots, a_{k}$. Although the result above can also be deduced from known results (see for example \cite{Ehr, SerOz}), the explicit form with exact values of coefficients is very useful in what follows.

Having the explicit form of the $P_{A}(2an+i), i\in\{0,\ldots, 2a-1\}, A=\{1,2,a\}$ one can obtain certain results concerning polynomial values taken by the partition function $P_{A}(n), n\in\N_{+}$. We start with the following simple

\begin{cor}
Let $a\in\N_{\geq 3}$ and put $A=\{1, 2, a\}$.
\begin{enumerate}
\item If $a\equiv 0\pmod*{2}$, then $P_{A}(2n)=P_{A}(2n+1)$ for each $n\in\N$.
\item If $a\equiv 1\pmod*{2}$, then
$$
P_{A}(n)=P_{A}(n+1)\;\Longleftrightarrow\; n=2j, j\in\left\{1,\ldots,\frac{a-3}{2}\right\}.
$$
\end{enumerate}
\end{cor}
\begin{proof}
\noindent (1) Let us put $a=2c, c\in\N_{\geq 2}$. The statement is an immediate consequence of the first formula from Theorem \ref{12aexpres}. Indeed, for each $i\in\{0,\ldots, 2c-1\}$ and $m\in\N$ we have the equality
$$
P_{A}(4c m+2i)=P_{A}(4c m+2i+1).
$$
Because of the equality $\{4cn+2i:\;n\in\N,\;i\in\{0,\ldots, 2c-1\}\}=2\N$ (the set of even non-negative integers) we get the result.

\noindent (2) Let us put $a=2c+1, c\in\N_{+}$. Because $\{2(2c+1)n+i:\;n\in\N,\;i\in\{0,\ldots, 4c+1\}\}=\N$ to get the solutions of $P_{A}(n)=P_{A}(n+1)$ it is enough to consider the solutions (in $m\in\N$) of equations $P_{A}(2(2c+1)m+2i)=P_{A}(2(2c+1)m+2i+1)$ or $P_{A}(2(2c+1)m+2i+1)=P_{A}(2(2c+1)m+2i+2)$. We consider the former equation first. If $i<c$ then we deal with the equation
$$
(2c+1)m^2+(c+2i+2)m+i+1=(2c+1)m^2+(c+2i+3)m+i+1,
$$
i.e., $m=0$ and $n=2i$ for $i=0,\ldots, c-1=\frac{a-3}{2}$. If $i=c$ then we work with the equation
$$
(2c+1)m^2+(c+2i+2)m+i+1=(2c+1)m^2+(c+2i+3)m+2i+2-c,
$$
i.e., $m=-1$ and we do not get any new solution.

The same analysis can be applied to the equation $P_{A}(2(2c+1)m+2i+1)=P_{A}(2(2c+1)m+2i+2)$ and we easily get that it has no solutions in $\N$. We omit the simple details.
\end{proof}

\begin{thm}\label{PA4}
Let $a\in\N_{\geq 3}$ and put $A=\{1,2, a\}$.
\begin{enumerate}
\item If $a\not\equiv 2\pmod*{4}$ then the Diophantine equation $P_{A}(m)=P_{4}(n)$ has infinitely many solutions in positive integers.
\item If $a\equiv 2\pmod*{4}$ then the Diophantine equation $P_{A}(m)=P_{4}(n)$ has only finitely many solutions in integers.
\end{enumerate}
\end{thm}
\begin{proof}
The general strategy of the proof is the following. The set of integer solutions of the equation $P_{A}(m)=P_{4}(n)$ is the sum of sets $U_{i,j}, i\in\{0,\ldots, 2a-1\}, j\in\{0,\ldots, 11\}$, where
$$
U_{i,j}=\{(2am+i,12n+j):\;P_{A}(2am+i)=P_{4}(12n+j),\;m, n\in\N\}.
$$
Thus, in order to show that the Diophantine equation $P_{A}(x)=P_{4}(y)$ has infinitely many solutions it is enough to prove that for some $i\in\{0,\ldots, 2a-1\}, j\in\{0,\ldots, 11\}$ the set $U_{i,j}$ is infinite. However, because $\op{deg}P_{A}(2am+i)$ has degree 2 and $\op{deg}P_{4}(12n+j)=3$, the set $U_{i,j}$ is infinite if and only if the discriminant (with respect to the variable $m$) of the polynomial
$$
F_{i,j,a}(m,n)=P_{A}(2am+i)-P_{4}(12n+j)
$$
is a square. Here, we treat $i, j$ and $n\in\N$ as  variables. Moreover, because $P_{4}(12n+j)$ is o degree three, then the discriminant
$$
G_{i,j,a}(n)= \op{Disc}_{m}(F_{i,j,a}(m,n))
$$
is a polynomial of degree three in the variable $n$. Thus, the value of $G_{i,j,a}(n)$ is square for infinitely many values of $n\in\N$ if and only if $G_{i,j,a}$, treated as a polynomial in the variable $n$, has double root. This, in turn, is equivalent with the vanishing of the discriminant of $G_{i,j,a}$. Summing up we get the following implication
$$
U_{i,j,a}\;\mbox{is infinite}\;\Longrightarrow\;H_{i,j,a}:=\op{Disc}_{n}(\op{Disc}_{m}(F_{i,j,a}(m,n)))=0.
$$
From our discussion it follows that we need to investigate the vanishing of $H_{i,j}$. However, before we will go one, let us note that a priori $U_{i,j,a}$ can be finite and the condition $H_{i,j,a}=0$ can be still satisfied (in other words we can not expect to have equivalence between the conditions above). This is clear. Due to the form of the polynomial $F_{i,j,a}$ to get an element of $U_{i,j,a}$ some additional congruence conditions need to be satisfied (which are not seen in discriminant computations). Indeed, let us take $a=9, i=0, j=3$, i.e.,
$$
F_{0,3,9}(m,n)=9 m^2+6 m-(12 n^3-24 n^2-15 n-2).
$$
Then $G_{0,3,9}(n)=108(n+1)(2 n+1)^2$. Thus $G_{0,3,9}(n)$ is a square if and only if $n=3u^2-1$. However, $F_{0,3,9}(m,3u^2-1)=(3m-18u^3+3u+1)(3m+18u^3-3u+1)$ and it is clear that our equation has no solutions.

After this discussion let us back to the proof of the statement.

To get the first part of our theorem we perform case by case analysis. If $a\equiv 0\pmod*{4}$, say $a=4s$ then $H_{2s-2, 3, 4s}=0$ and we have that
$$
G_{2s-2,3,4s}(n)=48(n+1)(2n+1)^2s.
$$
Thus, in order to make the above expression a square, we need to take $n=3su^2-1$. Then
$$
F_{2s-2,3,4s}(m, 3su^2-1)=s(-2 m+18 s u^3-3 u-1)(2 m+18 s u^3-3 u+1)
$$
and $m=\frac{1}{2} \left(18 s u^3-3 u-1\right)$. Summing up: if $u$ is odd positive integer then the numbers
\begin{align*}
x&=8sm+2s-2=2 \left(36 s^2 u^3-6 s u-s-1\right),\\
y&=12n+3=9 \left(4 s u^2-1\right)
\end{align*}
solve the equation $P_{A}(x)=P_{4}(y)$.

Because in the case $a\equiv 1, 3\pmod*{4}$ the reasoning goes in exactly the same way we present only the appropriate values of $i, j$ and the corresponding solutions $x, y$.

If $a\equiv 1\pmod*{4}$, i.e., $a=4s+1$ then we take $i=2s-1, j=2$ and $u$ positive odd number and get
\begin{align*}
x&=\frac{1}{2}(9(4s+1)^2u^3-3(4s+1)u-4(s+1)),\\
y&=(9s+4)u^2-7,
\end{align*}
positive integers solving the equation $P_{A}(x)=P_{4}(y)$.

If $a\equiv 1\pmod*{4}$, i.e., $a=4s+3$ then we take $i=s+1, j=0$ and $u$ positive odd number and get
\begin{align*}
x&=\frac{1}{2}(9(4s+3)^2u^3+(4s+3)u-2(2s+3)),\\
y&=3((4s+3)u^2-1),
\end{align*}
positive integers solving the equation $P_{A}(x)=P_{4}(y)$.

To get the second part of our theorem we need to investigate the vanishing of $H{i,j,4s+2}$. Because we have exact expression for $P_{A}(4(2s+1)m+2i)$ (which is equal to  $P_{A}(4(2s+1)m+2i)$) it is enough to consider $H_{i,j,4s+2}$ for $j=0, \ldots, 11$ as a polynomial in two variables: $s$ and $i$. Because we need to consider two cases $i\in\{0,\ldots, 2s\}$ and $i\in\{2s+1,\ldots, 4s-1\}$ we work with 24 polynomials $H_{i,j,4s+2}$. It is easy compute these polynomials. Each has the form
$$
C(i,j)(2s+1)^2Q_{j}(i,s)R_{j}(i,s),
$$
where $C(i,j)\in\Z$ and $Q_{j}, R_{j}$ are quadratic inhomogeneous polynomials. In each case the quadratic forms $Q_{j}, R_{j}$ has no integer zeros. Because in each case the reasoning is the same we present only one typical example. So let us suppose that $i\in\{0,\ldots, 2s\}$ and take $j=0$. Then
$$
H_{i,0,4s+2}=-27648(2s+1)^2(4(i^2+s^2)-4i(2 s-1)+3)(36(i^2+s^2)-36 i (2s-1)-4 s+25)
$$
and quick computation reveals that each factor is non-zero for $s\in\Z$ and $i\in\N$. Performing the same analysis for the rest of polynomials we get the statement of our theorem.
\end{proof}

The first part of the above result can be further generalized. In order to get the generalization we will need the following simple


\begin{thm}
Let $a\in\N_{\geq 3}$ and put $A=\{1, 2, a\}$. The Diophantine equation $y^2=P_{A}(x)$ has infinitely many solutions in positive integers.
\end{thm}
\begin{proof}
First we consider the case $a$ is not a square.

Let $a$ be even, i.e., $a=2c$ for some $c$. Take $i=0$ in the first formula in Theorem \ref{12aexpres}, i.e., we work with the Diophantine equation
$$
P_{A}(4cn)=2cn^2+(c+2)n+1=y^2.
$$
In order to show that the equation $P_{A}(4cn)=y^2$ has infinitely many integral solutions, we will follow the standard argument to parameterize (rational) solutions since we know that $(n,y)=(0,1)$ solves the equation. The lines through $(0,1)$ can be written as $y=mn+1.$ Therefore we get that $2cn^2+(c+2)n+1=\left(mn+1\right)^2,$ that is $n=0$ or
$$
n=\frac{c+2-2m}{m^2-2c}.
$$
Here $m=u/v$ is a rational parameter, so we have that
$$
n=\frac{(c+2)v^2-2uv}{u^2-2cv^2}.
$$
For our assumption, $a=2c$ is not a square, then we consider the Pell-equation $u^2-2cv^2=1$ and denote the sequence of positive integer solutions by $(u_k,v_k).$ In this case it follows that $n=(c+2)v_k^2-2u_kv_k$ and $y=(c+2)u_kv_k-2u_k^2.$


Let us now consider the case with $a=2c+1$ odd. Applying the second formula in Theorem \ref{12aexpres} with $i=0$, we work with the Diophantine equation
$$
P_{A}(2(2c+1)n)=(2c+1)n^2+(c+2)n+1=y^2.
$$
Again, in order to show that the equation $P_{A}(4cn)=y^2$ has infinitely many integral solutions, we follow the standard argument. The pair $(n,y)=(0,1)$ solves the equation. The lines through $(0,1)$ can be written as $y=mn+1.$ Therefore, we get that $(2c+1)n^2+(c+2)n+1=\left(mn+1\right)^2,$ that is $n=0$ or
$$
n=\frac{c+2-2m}{m^2-2c-1}.
$$
Here $m=u/v$ is a rational parameter, so we have that
$$
n=\frac{(c+2)v^2-2uv}{u^2-(2c+1)v^2}.
$$
For our assumption, $a=2c+1$ is not a square, then we consider the Pell-equation $u^2-(2c+1)v^2=1$ and denote the sequence of positive integer solutions by $(u_k,v_k).$ In this case it follows that $n=(c+2)v_k^2-2u_kv_k$ and $y=(c+2)u_kv_k-2u_k^2$.

Summing up: we proved that if $a$ is not a square then the Diophantine equation $y^2=P_{A}(x)$ has infinitely many solutions in positive integers.

It remains to deal with the case when $a$ is a square (even or odd). We follow similar lines, so we only provide details in case of $a=4t^2$, that is when $a$ is an even square.
Again, using the first formula from Theorem \ref{12aexpres} with $c=2t^2, i=2t^2-2<c-1$ we get that
$$
P_{A}(8t^2n+2t^2-2)=4t^2n^2+4t^2n+t^2=t^2(2n+1)^2,
$$
and for each $n\in\N_{+}$ the number $P_{A}(8t^2n+2t^2-2)$ is a square and the Diophantine equation $y^2=P_{A}(x)$ has infinitely many solutions in positive integers.

If $a=(2t+1)^2$ is an odd square, then using the second formula from Theorem \ref{12aexpres} with $c=2t(2t+1)$ and $i=2t^2-2$ we find that
$$
P_{A}(2(2t+1)^2n+2t^2-2)=((2t+1)n+t)^2.
$$
For each $n\in\N$ the number $P_{A}(2(2t+1)^2n+2t^2-2)$ is a square and our theorem is proved.
\end{proof}

\begin{thm}\label{12a12b}
Let $a, b\in\N_{\geq 3}, a<b$ such that $a,b$ are divisible by 4 and either $a/2$ or $b/2$ is not a square. Put $A=\{1, 2, a\}, B=\{1, 2, b\}$. The Diophantine equation $P_{A}(x)=P_{B}(y)$ has infinitely many solutions in positive integers.
\end{thm}
\begin{proof}
Let $a=2s$ and $b=2t.$ It follows from Theorem \ref{12aexpres} that
\begin{eqnarray*}
P_A(4sn)&=&2sn^2+(s+2)n+1,\\
P_B(4tm)&=&2tm^2+(t+2)m+1.
\end{eqnarray*}
Suppose that $a/2=s$ is not a square.
We have the solution $(0,0)$ of the equation $P_A(4sn)=P_B(4tm)$ so we write $n=km$ for some rational number $k=u/v.$
Solving the equation for $m$ provides that either $m=0$ or
$$
m=\frac{(t+2)v^2-suv-2uv}{2su^2-2v^2}.
$$
Since $s$ and $t$ are even integers the numerator is divisible by 2 and we obtain the expression
$$
m=\frac{(s+2)/2uv-(t+2)/2v^2}{v^2-su^2}.
$$
The integer $s$ is not a square, hence we consider the sequence of positive solutions $(u_k,v_k)$ of the Pell-equation $v^2-su^2=1.$ For these solutions we have $m=(s+2)/2u_kv_k-(t+2)/2v_k^2$ and $n=(s+2)/2u_k^2-(t+2)/2u_kv_k.$
\end{proof}

In view of theorem above we formulate the following conjecture.
\begin{conj}
Let $a, b\in\N_{\geq 3}, a<b$ and put $A=\{1, 2, a\}, B=\{1, 2, b\}$. The Diophantine equation $P_{A}(x)=P_{B}(y)$ has infinitely many solutions in positive integers.
\end{conj}

Let us note that to prove the above conjecture it is enough to find a pair $(i,j)$ of integers such that $i\in\{0,\ldots, 2a-1\}, j\in\{0,\ldots, 2b-1\}$ and the Diophantine equation $P_{A}(2ax+i)=P_{B}(2by+j)$ has infinitely many solutions in integers. Although for any fixed values of $a, b$ it is easy to find suitable $i, j$ we were unable to get the general result.

\section{Remarks on the Diophantine equation $y^2=P_{A}(x)$}\label{sec5}

A difficult and still unsolved question is whether the number $p(n)$ can be a perfect power. Let us recall that $p(n)$ counts the number of all partitions of $n$, i.e.,
$$
\prod_{n=1}^{\infty}\frac{1}{1-x^{n}}=\sum_{n=0}^{\infty}p(n)x^{n}.
$$
In other words, we do not know any example of $n\geq 2$ such that $y^k=p(n)$ for some $k\in\N_{\geq 2}$. In fact Zhi-Wei Sun conjectured that the equation $y^{k}=p(n)$ has no solutions in positive integers $n, y, k$ with $k\geq 2$. Let us also note that Alekseyev checked that there are no solutions with $n\leq 10^{8}$ \cite{mathover}.

A question arises whether some results concerning the equation $y^2=P_{k}(x)$ can be proved for some values of $k\in\N_{+}$. We know that the Diophantine equation $y^2=P_{k}(x)$ has infinitely many solutions in positive integers for $k\leq 4$. Indeed, to get the result for $k=3$ it is enough to back to the explicit form of $P_{3}(6n+i)$ presented in the proof of Theorem \ref{P3P4}. It is easy to see that for $i\in\{0, 1, 4, 5\}$ the Diophantine equation
$$
y^2=P_{3}(6n+i)
$$
has infinitely many solutions in positive integer. For example, if $i=4$ we deal with the equation $y^2=(n+1)(3n+4)$. Thus, if $(u_{k}, v_{k})$ is a solution of the Pell equation $v^2-3u^2=1$ then the pair $(n, y)=(u_{k}^2-1, u_{k}v_{k})$ solves our equation.

If $k=4$ then again we back to the explicit form of $P_{4}(6n+2i+1)$ for $i=0, 1, 2$ and $P_{4}(12n+2i)$ for $i=0, 1, \ldots, 5$. A quick inspection reveals that the equation $y^2=P_{4}(6n+2i+1)$ has infinitely many solutions for $i=1, 2$. For example, if $i=1$ then it is enough to take $n=6u^2-2$ and $y=3u(6u^2-1)$. Similarly, it is easy to see that equation $y^2=P_{4}(12n+2i)$ has infinitely many solutions for $i=3, 4$.

The first non-trivial problem is characterization of the positive integer solutions of the Diophantine equation $y^2=P_{5}(x)$. We prove the following

\begin{thm}
The equation $y^2=P_{5}(x)$ has only finitely many solutions in positive integers. More precisely, the pair $(x,y)$ is a solution if and only if $(x,y)=(1,1), (2027, 77129)$.
\end{thm}
\begin{proof}
We have 60 curves of the form $y^2=P_{5}(60n+i), i\in\{0,\ldots, 59\}$. If $i\in\{5, 20, 25, 40\}$ the corresponding quartic has no $\mathbb{Q}_{5}$-rational points, and thus has no rational points at all.

Similarly as in the proof of Theorem \ref{P3P4} and Theorem \ref{P3P5}, we apply the procedure \texttt{IntegralQuarticPoints()} to determine all integral solutions in the remaining 56 cases. In particular, the solution $(1, 1)$ comes from the equation $y^2=P_{5}(60n+1)$ with $n=0$. The solution $(2027, 77129)$ comes from the solution $(n,y)=(33, 77129)$ of the equation $y^2=P_{5}(60n+47)$. The procedure works well, except in 6 special cases.  Here we do not get any error message like in the special cases appearing in the proof of Theorem \ref{P3P5}, but warnings about time-consuming final  enumerations. The 6 problematic polynomials correspond to $i\in\{21, 24, 48, 51, 54, 57\}$.
The equations (up to multiplication by $16$) corresponding to $i=21, 24$ give the following equations
\begin{eqnarray*}
5y^2&=&u(36u^3+108u^2+34u+12),\quad u=5(2n+1),\\
5y^2&=&u(36u^3-108u^2+34u-12),\quad u=2(5n+3),
\end{eqnarray*}
respectively. Hence we need to resolve the following elliptic equations
$$
Y^2=X^3+108\delta X^2+3384\delta^2X+15552\delta^3,
$$
where $\delta$ divides 60. We only get the trivial solution given by $u=0$.

For $i\in\{48, 51, 54, 57\}$ after the substitution $u=2(n+1)$ we get the following quartic equations
\begin{eqnarray*}
y^2&=&u(4500u^3-2700u^2+470u-12),\\
y^2&=&u(4500u^3-900u^2-70u+4),\\
y^2&=&u(4500u^3+900u^2-70u-4),\\
y^2&=&u(4500u^3+2700u^2+470u+12).
\end{eqnarray*}
We obtain elliptic equations in a similar way as before, so we omit details. It turns out that we get only the trivial solution with $u=0$ from these cases.
\end{proof}

The case of the equation $y^2=P_{6}(x)$ is far more difficult. To get the solutions we need to consider 60 genus 2 curves
$$
C_{i}:\;y^2=P_{6}(60n+i),\quad i=0, \ldots, 59.
$$
Let $J_{i}=\op{Jac}(C_{i})$ be the Jacobian of the curve $C_{i}$ and by $r_{i}$ denote the rank of $J_{i}$. We checked that $r_{i}\leq 5$ for $0\leq i\leq 59$.

\begin{equation*}
\begin{array}{|l||l|}
\hline
  r & \mbox{values of}\;i\;\mbox{such that}\;r_{i}\leq r \\
\hline
  0 & 3, 14, 34, 47, 50, 51, 55, 59 \\
  1 & 18, 22, 27, 32, 35, 38, 41, 43, 44, 45, 46, 54  \\
  2 & 0, 7, 8, 9, 15, 23, 24, 25, 26, 28, 29, 30, 33, 36, 37, 39, 40, 42, 49, 52, 53, 57, 58 \\
  3 & 2, 5, 6, 11, 17, 31, 48 \\
  4 & 4, 10, 13, 16, 19, 20, 21, 56 \\
  5 & 1, 12\\
\hline
\end{array}
\end{equation*}
\begin{center}
Table. Upper bounds for the $\Q$-rank of the Jacobian $J_{i}$ of the curve $C_{i}:\;y^2=P_{6}(60x+i)$ for $i=0, \ldots, 59$.
\end{center}

It is curious that the polynomial $P_{6}(60n+i)$ is reducible (in the ring $\Q[n]$) for $i\in\{40,\ldots, 59\}$ and thus, instead of working with genus two curve we need to play with certain curves of the type $y^2=Q_{i}(x)$, where $Q_{i}$ is a quartic polynomial.

If the rank of the Mordell-Weil group is less than the genus of the curve, that is 2 in these cases, then classical Chabauty's method \cite{Chab} may be applied to determine all rational points on the hyperelliptic curves. If the rank is greater than or equal to 2, then there are two different approaches to compute the set of integral points on the curves (see \cite{BMSST, GallHyp}). The difficulty is that one needs a Mordell-Weil basis. Among the above curves there are some for which we were not able to obtain such bases, these are as follows
$C_i$ with
$$
i\in\{15,16,23,24,27,28,29,31,32,33,35,36,38,39
\}.$$
The most interesting one may be the hyperelliptic curve given by
$$
y^2=12x^5 + 1125x^4 + 41960x^3 + 778050x^2 + 7171020x + 26276400,
$$
which, as computed with the help of {\sc Magma}, is the minimal model of the curve $C_{27}:~~y^2=P_6(60n+27).$ In this case the rank is 1, however we were unable to found a generator of the Mordell-Weil group.

We finish with the following

\begin{conj}
Let $n\in\N_{\geq 6}$. The only positive integer solution of the Diophantine equation $y^2=P_{n}(x)$ is $x=y=1$.
\end{conj}

\bigskip

Motivated by the results above one can ask a more general

\begin{ques}
Let $A\subset \N_{+}$ and suppose that the Diophantine equation $y^2=P_{A}(x)$. How large the number $\#A$ can be?
\end{ques}

In case of $\#A=5$ there is a large number of sets such that $P_{A}(L_{A}n+i)$ is a square of a polynomial in $n$. More precisely, with the constraint $\op{max}(A)\leq 15$, there are exactly 119 different pairs $(A,i)$ such that $P_{A}(L_{A}n+i)$ is a square of a polynomial with integer coefficients. For example, if $A=\{1, 2, 8, 10, 15\}$, then $L_{A}=120$ and for $i=1, 11, 41, 43, 73, 83, 91, 113$ we have $P_{A}(L_{A}n+i)$ is a square of a polynomial. In particular,
$$
P_{A}(120n+1)=(4n+1)^2(15n+1)^2.
$$
In the table below we collect data concerning our search.
\begin{equation*}
\begin{array}{|l|l|l|}
\hline
A & L_{A} & i \\
\hline
 \{1,2,8,10,15\} & 120 & 1,11,41,43,73,83,91,113 \\
 \{1,4,5,10,12\} & 60 & 12,16,36,52 \\
 \{1,4,8,9,12\} & 72 & 1,13,19,25,37,43,49,61,67 \\
 \{1,5,6,8,10\} & 120 & 2,8,13,17,32,37,53,58,73, 77,82,88,97,98,112,113\\
 \{2,3,7,8,14\} & 168 & 32,102,144,158 \\
 \{2,4,5,6,10\} & 60 & 12,16,17,21,36,41,52,57 \\
 \{3,4,6,9,12\} & 36 & 3,7,11,27,31,35 \\
 \{3,5,6,9,15\} & 90 & 18,23,24,28,29,34,54,59,64,78,83,88 \\
 \{4,5,6,12,15\} & 60 & 27,51 \\
 \{4,7,9,12,14\} & 252 & 58,64,142,148,226,232 \\
 \{5,6,8,9,10\} & 360 & 8,29,53,74,89,98,104,113,128,149,173,194,209,\\
                &     & 218,224,233,248,269,293,314,329,338, 344,353 \\
 \{5,7,9,14,15\} & 630 & 47,113,173,197,257,323,383,407,467,533,593,617 \\
 \{7,8,10,14,15\} & 840 & 182,212,364,422,574,604,812,814 \\
\hline
\end{array}
\end{equation*}
\begin{center}
Table. The sets $A$ such that $\#A=5, \op{max}(A)\leq 15$ and there is an $i\in\{0,\ldots, L_{A}-1\}$ such that $P_{A}(L_{A}n+i)$ is a square of a polynomial in $\Z[n]$,
\end{center}

In case of $\#A=6$ there is a large number of sets such that $P_{A}(L_{A}n+i)$ is a square of a polynomial (with rational coefficients) in $n$ times a linear factor (note that this is only possibility to get infinitely many square values). However, in each case the values of a corresponding linear factor nor the value of $P_{A}(L_{A}n+i)$ can be a square of an integer.

\
We were able to find only the one set $A$ with 7 elements, $\op{max}(A)\leq 10$ and such that $y^2=P_{A}(x)$ has infinitely many solutions in positive integers. More precisely, if $A=\{1,2,4,5,8,9,10\}$ then
\begin{align*}
P_{A}(360n+95)&=25(3n+1)^2(18n+5)^2(36n+13)(40n+13),\\
P_{A}(360n+226)&=25(3n+2)^2(18n+13)^2(36n+23)(40n+27).
\end{align*}
One can easily check that the factor $(36n+13)(40n+13)$ is a square infinitely often. The smallest values of $n$ which makes this factor a square, are $n=0, 494, 712842, \ldots$.  However, the factor $(36n+23)(40n+27)$ takes square values for infinitely many values negative values of $n$ and thus is not of interests for us.

\section{Problems, questions and conjectures}\label{sec6}


Besides the conjectures stated in previous sections, we formulate now several question and conjectures which hopefully will stimulate further research.

\begin{ques}
Let $k\in\N_{\geq 2}$ and $f\in\Z[x]$ be given. Does there exist an ascending sequence of sets $A_{2}=\{a_{1}, a_{2}\}\subset\ldots \subset A_{k}=\{a_{1},\ldots, a_{k}\}\ldots \subset\N_{+}$ such that the Diophantine equation $P_{A_{k}}(x)=f(y)$ has at least $C_{k,f}$ solutions in positive integers and $C_{k,f}\rightarrow +\infty$ as $k\rightarrow +\infty$?
\end{ques}

Let us observe that without the condition $C_{k,f}\rightarrow +\infty$ the question is not difficult. Indeed, let us take $A_{2}=\{a_{1}, a_{2}\}\subset\N_{+}$ and suppose that $\gcd(a_{1},a_{2})=1$. As we already proved in Theorem \ref{a1a2} the Diophantine equation $P_{A_{2}}(x)=f(y)$ has infinitely many solutions in positive integers. If $C\in\N$ is fixed let us take an increasing sequence $\{a_{3}, a_{4},\ldots, a_{k}\}$ of positive integers such that $a_{3}$ is grater then the smallest integer $N$ such that there is at least $C$ values of $x$ for which there is an integer $y$ satisfying $P_{A_{2}}(x)=f(y)$. Then, for $A_{k}=\{a_{1}, a_{2}, a_{3},\ldots, a_{k}\}$ the Diophantine equation
$$
P_{A_{k}}(x)=f(y)
$$
has at least $C$ solutions in positive integers. Indeed, this is simple consequence of the recurrence relation satisfied by the sequence $\{P_{A_{k}}(n)\}_{n\in\N}$. Indeed, because $P_{A_{k}}(n)=P_{A_{k-1}}(n)$ for $n<a_{k}$, then $P_{A_{k}}(n)=P_{A_{2}}(n)$ for $n<\min\{a_{3},\ldots, a_{k}\}=a_{3}$ and hence the result.

\bigskip

We proved that the equation $P_{3}(x)=P_{5}(x)$ has only finitely many solutions in positive integers and it is quite natural to ask whether there are $A, B$ satisfying $\#A=3, \#B=5$, such that the equation $P_{A}(x)=P_{B}(y)$ has infinitely many solutions in positive integers. To get the result in this direction we will need the following.

\begin{lem}\label{specb}
Let $b\in\N_{\geq 4}$ and put $B=\{1, 2, 3, 4, b\}$.
\begin{enumerate}
\item If $b=4(6k+1), j=3(8k-1)$ for some $k\in\N_{+}$, then $P_{B}(3bn+j)=(3n+2)((6k+1)n+2k)Q_{1}(k,n)$, where
$$
Q_{1}(k,n)=3(6k+1)^2n^2+2(9k+1)(6k+1)n+6k(4k+1).
$$
\item If $b=4(6k+5), j=24k+13$ for some $k\in\N_{+}$, then $P_{B}(3bn+j)=(3n+1)((6k+5)n+4k+3)Q_{2}(k,n)$, where
$$
Q_{2}(k,n)=3(6k+5)^2n^2+2(6k+5)(9k+7)n+24k^2+36k+1).
$$
\item If $b=4(12k+2), j=48k+1$ for some $k\in\N_{+}$, then $P_{B}(3bn+j)=(3n+1)(2(6k+1)n+8k+1)Q_{3}(k,n)$, where
$$
Q_{3}(k,n)=12(6k+1)^2n^2+2(6k+1)(36k+5)n+96k^2+24k+1.
$$
\item If $b=4(12k+10), j=48k+1$ for some $k\in\N_{+}$, then $P_{B}(3bn+j)=(3n+2)(2(6k+5)n+4k+3)Q_{4}(k,n)$, where
$$
Q_{4}(k,n)=12(6k+5)^2n^2+2(6k+5)(36k+29)n+3(4k+3)(8k+7).
$$
\end{enumerate}
\end{lem}
\begin{proof}
Let us note that the sequence $\{P_{B}(n)\}_{n\in\N}$ satisfies the following recurrence relation
$$
P_{B}(n)=\begin{cases}\begin{array}{ll}
                        P_{A}(n),             & n<b, \\
                        P_{B}(n-b)+P_{A}(n), & b\leq n,
                      \end{array}
\end{cases}
$$
where $A=\{1, 2, 3, 4\}$. We know the polynomial expressions for $P_{A}(12n+i), i\in\{0,\ldots, 11\}$ and that $P_{B}(L_{b}n+j), j\in\{0,\ldots, L_{b}-1\}$, where $L_{b}=\op{LCM}(1,2,3,4,b)$, is a polynomial of degree 4 with rational coefficients. Using induction one can obtain the expression for the polynomials of interests. We omit tiresome details.
\end{proof}

\begin{thm}
Let $a\in\N_{\geq 3}, b\in\N_{\geq 4}$ and put $A=\{1, 2, a\}, B=\{1, 2, 3, 4, b\}$. If $a\equiv 1, 2, 5, 7, 11, 10\pmod*{12}$ and $b=4a$, then the Diophantine equation $P_{A}(x)=P_{B}(y)$ has infinitely many solutions in positive integers.
\end{thm}
\begin{proof}
Note that if $a\equiv 1, 2, 5, 7, 11, 10\pmod*{12}$, then $a$ can be written in one of the following form: $a=6k+1, a=6k+5, a=12k+2$ $a=12k+10$. Thus, in each case, the value of $b=4a$ is exactly the value of $b$ considered in Lemma \ref{specb}. Following the idea of proof of Theorem \ref{PA4} we present the values of $i, j$ such that the polynomial $P_{A}(2am+i)-P_{B}(3bn+j)$ is reducible and the coefficient in of the linear factor (in $m$) near $m$ is equal to 1.

Let $a=6k+1, b=4a, i=11k, j=3(8k-1)$. Then $P_{A}(2am+i)-P_{B}(3bn+j)=R_{1}(m,n)R_{2}(m,n)$, where
\begin{align*}
R_{1}(m,n)&=m-3(6k+5)n^2-2(9k+7)n-4k+1,\\
R_{2}(m,n)&=(6k+1)m+3(6k+1)^2n^2+2(6k+1)(9k+1)n+24k^2+12k+1.
\end{align*}
Thus, if $m=3(6k+5)n^2+2(9k+7)n+4k-1$ then $P_{A}(2am+i)=P_{B}(3bn+j)$ and our equation has infinitely many solutions.

Because in each case we proceed in the same way we present only the values of $a, i, b, j$ and the corresponding solution for $m$.

If $a=6k+5, i=7k+4, b=4(6k+4), j=24k+13$, then
$$
m=3(6k+5)n^2+2(9k+7)n+2(2k+1).
$$

If $a=2(6k+1), i=14k, b=8(6k+1), j=48k+1$, then
$$
 m=6(6k+1)n^2+(36k+5)n+8k.
$$

If $a=2(6k+5), i=2(11k+8), b=8(6k+5), j=3(16k+11)$, then
$$
m=6(6k+5)n^2+(36k+29)n+8k+5.
$$

\end{proof}

We proved that for many choices of sequences $A, B$, the corresponding Diophantine equation $P_{A}(x)=P_{B}(y)$ has infinitely many solutions in positive integers. However, in each case under consideration we had $\op{min}\{\#A,\#B\}\leq 3$. This observation lead us to the following.

\begin{ques}
Let $A, B\subset \N_{+}$. Let us suppose that the Diophantine equation $P_{A}(x)=P_{B}(y)$ has infinitely many (non-trivial) solutions in positive integers. How large the number $\op{min}\{\#A, \#B\}$ can be?
\end{ques}

Let us explain what a trivial solution means. More precisely, if for example $A=\{1, pa_{2},...,pa_{k}\}$ then if $P_{A}(pn)$ is a non-zero, then in each representation
$$
1\cdot x_{1}+\sum_{i=2}^{k}pa_{i}x_{i}=pn
$$
we need to have $p|x_{1}$ and thus we get a representation
$$
1\cdot y_{1}+\sum_{i=2}^{k}a_{i}x_{i}=n.
$$
It is clear that this mapping can be reversed. Thus, by taking $B=\{1, a_{2}, ..., a_{k}\}$ we have the boring identity $P_{A}(pn)=P_{B}(n)$.

Thus, in regards to question above, we considered equations of the form $P_A(x)=P_B(y),$ where $A,B$ are sets having 5 elements from $\{1,2,\ldots,10\}$ and one of the elements is 1. We searched for reducible polynomials $P_A(x)-P_B(y)$ having a linear or quadratic factor. We implemented a parallel algorithm and used SageMath on a machine having 16 cores. It took about 10 hours to determine the appropriate polynomials. There are 44982 such cases. Among these polynomials we looked for examples providing infinitely many integral solutions. To reduce the time of computation a timeout was set to be 60 seconds. There are 392 cases for which the 60 seconds were not sufficient to compute the result. There are 2338 quadratic equations that yield infinitely many integral solutions and 2100 linear equations that provide parametric solutions. However, even in the case of reducibility we sometimes get factors without positive integer solutions. We present several examples.

Let $A=\{1,2,4,5,6\}$ and $B=\{1,4,6,9,10\}.$ Here we obtain that $P_A(60m+22)-P_B(180n+111)$ is,
up to a constant factor, equal to $f_{1}(m,n)f_{2}(m,n)$, where
\begin{eqnarray*}
f_{1}(m,n)&=&150m^2 + 450n^2 + 155m + 630n + 259,\\
f_{2}(m,n)&=&30m^2 - 90n^2 + 31m - 126n - 36.
\end{eqnarray*}
The equation $f_{1}(m,n)=0$ has no solution modulo 5. The equation $f_{2}(m,n)=0$ has infinitely many integral solutions. However, all are negative and are not of interest for us.

As a second example consider $A=\{1,2,4,6,10\}$ and $B=\{1,2,5,6,8\}.$ We get that $P_A(60m+17)-P_B(120n+17)$ is, up to a constant factor, equal to $g_{1}(m,n)g_{2}(m,n)g_{3}(m,n)$, where
\begin{eqnarray*}
g_{1}(m,n)&=&m-2n\\
g_{2}(m,n)&=&15m + 30n + 14,\\
g_{3}(m,n)&=&75m^2 + 300n^2 + 70m + 140n + 31.
\end{eqnarray*}
We obtain infinitely many integral solutions from the equation $g_{1}(m,n)=0$ (however, these are trivial solutions). The other two equations have no solutions modulo 5.

As a third example let $A=\{1,2,3,4,6\},B=\{1,2,4,5,10\}.$ It follows that $P_A(12m+1)-P_B(20n+1)=1/6h_{1}(m,n)h_{2}(m,n)$, where
\begin{eqnarray*}
h_{1}(m,n)&=&6m^2 + 10n^2 + 9m + 12n + 5\\
h_{2}(m,n)&=&6m^2 - 10n^2+ 9m - 12n.
\end{eqnarray*}
The equation $h_{1}(m,n)=0$ can be written as
$$
15(36m+27)^2+(180n+108)^2=6399,
$$
and it follows that the only integral solution is given by $(m,n)=(-1,-1).$ The equation $h_{2}(m,n)=0$ has infinitely many positive integral solutions, the two smallest being $(m,n)=(2928,2268), (11252256,8715960)$.

For given $A\in\N_{+}$ the function $P_{A}(n)$ has a dual nature: from one side it is a quasi-polynomial and hance an algebraic object. On the other side $P_{A}(n)$ is counting function and thus live in a realm of combinatorics. In this paper we mainly operated on the former side. Thus, it is natural to state the following general question.

\begin{prob}
Let $A, B\subset\N_{+}$. Does there exist combinatorial conditions on $A$ and $B$ which guarantees non-existence (or finiteness) of integral solutions of the Diophantine equation $P_{A}(x)=P_{B}(y)$?
\end{prob}

It is clear that the above problem can be stated in a grater generality. More precisely, we can ask whether there are some combinatorial conditions which guarantee that for not necessarily finite sets $A_{1}, A_{2}$, and corresponding properties $\cal{W}_{1}, \cal{W}_{2}$, the equation $p_{A_{1}}(\cal{W}_{1},x)=p_{A_{2}}(\cal{W}_{2},y)$ has only finitely many solutions in positive integers.

As we mentioned above, the partition functions count combinatorial objects. Thus, equality between different partition functions at certain integer arguments is equivalent with the statement that certain finite sets have the same number of elements. This suggest the following

\begin{prob}
Let $A, B\subset\N_{+}$ and suppose that the Diophantine equation $p_{A}(x)=p_{B}(y)$ has infinitely many solutions in integers. Moreover, let $x=\phi(n), y=\psi(n)$ be parametrization of one (of possibly many) infinite part of the solution set, i.e., $p_{A}(\phi(n))=p_{B}(\psi(n))$ for each $n\in\N_{+}$. Describe the bijection (in combinatorial or other way) between the sets $\op{Part}(\phi(n))=\op{Part}(\psi(n))$.
\end{prob}

Motivated by our findings presented in Theorem \ref{P3P5} and related results we formulate the following
\begin{conj}
Let $m, n\in\N_{+}$. If $(m,n)\neq (3, 4)$ and $3\leq m<n$, then the Diophantine equation $P_{m}(x)=P_{n}(y)$ has only finitely many solutions in non-negative integers.
\end{conj}

\begin{rem}
{\rm Let us note that from the recurrence relation satisfied by the sequence $\{P_{m}(k)\}_{k\in\N}$, i.e.,
$$
P_{m}(k)=P_{m-1}(k), k<m, \quad P_{m}(k)=P_{m-1}(k)+P_{m}(k-m),\;k\geq m,
$$
we know that the equation $P_{m}(x)=P_{n}(y)$ has trivial solutions $x=y=i, i\leq m$. So, it is reasonable to consider the set
$$
C_{m,n}:=\{(x, y)\in\Z\times\Z:\;P_{m}(x)=P_{n}(y)\wedge y\geq n\}.
$$
We believe that much stronger property is true, i.e.,
$$
\bigcup_{\min\{m, n\}\geq 3, m<n, (m,n)\neq (3,4)}C_{m,n}(\N)<+\infty.
$$

}
\end{rem}

\begin{acknowledgement}
The authors are grateful to Nikolaos Tzanakis for his ideas to complete the proof of Theorem \ref{P3P5}. We are also grateful for an anonymous referee for remarks which led to  improving the presentation.
\end{acknowledgement}

 \newcommand{\noop}[1]{} \def\cprime{$'$}


 \end{document}